\documentclass[a4paper,oneside,10pt]{article}%
\usepackage{amsmath}
\usepackage{amsfonts}
\usepackage{amssymb}
\usepackage{graphicx}%
\providecommand{\U}[1]{\protect\rule{.1in}{.1in}}

\newtheorem{theorem}{Theorem}[section]

\newtheorem{corollary}[theorem]{Corollary}

\newtheorem{definition}[theorem]{Definition}
\newtheorem{assumption}[theorem]{Assumption}

\newtheorem{lemma}[theorem]{Lemma}

\newtheorem{remark}[theorem]{Remark}

\newenvironment{proof}[1][Proof]{\noindent \textbf{#1.} }{\  \rule{0.5em}{0.5em}}
\numberwithin{equation}{section}

\begin{document}

\title{ The Dupire derivatives and Fr\'{e}chet derivatives on continuous pathes \thanks{This work was supported by National Natural Science Foundation of China (No. 11171187, No. 10871118 and No. 10921101); Supported by the Programme of Introducing Talents of Discipline to Universities of China (No.B12023);
Supported by Program for New Century Excellent Talents in University of China.}}
\author{Shaolin Ji\thanks{Institute for Financial Studies and Institute of
Mathematics, Shandong University, Jinan, Shandong 250100, PR China
(Jsl@sdu.edu.cn, Fax: +86 0531 88564100).}
\and Shuzhen Yang\thanks{School of mathematics, Shandong University, Jinan,
Shandong 250100, PR China. (yangsz@mail.sdu.edu.cn). }}
\date{}
\maketitle

\textbf{Abstract:} In this paper, we study the relation between Fr\'{e}chet
derivatives and Dupire derivatives, in which the latter are recently
introduced by Dupire {\normalsize \cite{Dupire.B}}. After introducing the
definition of Fr\'{e}chet derivatives {\normalsize for non-anticipative
functionals, we prove that }the Dupire derivatives and the extended
Fr\'{e}chet derivatives are coherent on continuous pathes.


\section{Introduction}

Recently Dupire {\normalsize \cite{Dupire.B} introduced the functional
It\^{o}'s calculus, which was further developed in Cont and Fourni
\cite{Cont-1}-\cite{Cont-3}. }The key idea of Dupire
{\normalsize \cite{Dupire.B} is to introduce the new "local" derivatives,
i.e., horizontal derivative and vertical derivative for non-anticipative
processes. }Inspired by Dupire's work, Peng and Wang \cite{Peng S 3} obtained
a nonlinear Feynman-Kac formula for classical solutions of path-dependent PDEs
in terms of non-Markovian Backward stochastic differential equations (BSDEs
for short). The viscosity solutions of path-dependent PDEs are also studied in
\cite{Ekren} and \cite{Peng S 2} under this new framework. All these results
show that Dupire derivative is an important tool {\normalsize to deal with
functionals of continuous semimartingales. }

The aim of this paper is to establish the relation between Dupire derivatives
and Fr\'{e}chet derivatives. Note that the Dupire derivative is a "local" one,
in the sense that it is defined by perturbing the endpoint of a given current
path. Compared with the Dupire derivative, the Fr\'{e}chet derivative is
defined by perturbing the whole path. Thus, it seems difficult to find the
relationship between them.

To overcome the above difficulty, we introduce the definition of Fr\'{e}chet
derivatives {\normalsize for non-anticipative functionals. Inspired by
Mohammed's work about }stochastic functional differential equations with
bounded memory (see \cite{Mohammed1} and \cite{Mohammed}), we study the weakly
continuous linear and bilinear extensions of the Fr\'{e}chet derivatives. By
means of an auxiliary stochastic functional system, we show that the Dupire
derivatives and the extended Fr\'{e}chet derivatives are coherent on
continuous pathes.

This paper is organized as follows. In section 2, we present some fundamental
results of the\ Duprie derivatives and define the Fr\'{e}chet derivatives
{\normalsize for non-anticipative functionals}. Furthermore, the unique
extensions of the Fr\'{e}chet derivatives are obtained. In section 3, under
mild assumptions, we prove that the Duprie derivatives and the extended
Fr\'{e}chet derivatives are equal on continuous pathes.

\section{Preliminaries}

\subsection{The Dupire derivatives}

The following notations and tools are mainly from Dupire \cite{Dupire.B}. Let
$T>0$ be fixed. For each $t\in\lbrack0,T]$, we denote by $\Lambda_{t}$ the set
of c\`{a}dl\`{a}g $\mathbb{R}^{d}$-valued functions on $[0,t]$, and $C$ is the
set of continuous functions on $[0,T]$. For each $\gamma(\cdot)\in\Lambda_{T}$
the value of $\gamma(\cdot)$ at time $s\in\lbrack0,T]$ is denoted by
$\gamma(s)$. Thus $\gamma(\cdot)=\gamma(s)_{0\leq s\leq T}$ is a
c\`{a}dl\`{a}g process on $[0,T]$ and its value at time $s$ is $\gamma(s)$.
The path of $\gamma(\cdot)$ up to time $t$ is denoted by $\gamma_{t}$, i.e.,
$\gamma_{t}=\gamma(s)_{0\leq s\leq t}\in\Lambda_{t}$. We denote $\Lambda
=\bigcup_{t\in\lbrack0,T]}\Lambda_{t}$. For each $\gamma_{t}\in\Lambda$ and
$x\in\mathbb{R}^{d}$ we denote by $\gamma_{t}(s)$ the value of $\gamma_{t}$ at
$s\in\lbrack0,t]$ and $\gamma_{t}^{x}:=(\gamma_{t}(s)_{0\leq s<t},\gamma
_{t}(t)+x)$ which is also an element in $\Lambda_{t}$.

Let $\langle\cdot,\cdot\rangle$ and $|\cdot|$ denote the inner product and
norm in $\mathbb{R}^{n}$. We now define a distance on $\Lambda$. For each
$0\leq t,\bar{t}\leq T$ and $\gamma_{t},\bar{\gamma}_{\bar{t}}\in\Lambda$, we
denote%
\[%
\begin{array}
[c]{l}%
\Vert\gamma_{t}\Vert:=\sup\limits_{s\in\lbrack0,t]}|\gamma_{t}(s)|,\\
\Vert\gamma_{t}-\bar{\gamma}_{\bar{t}}\Vert:=\sup\limits_{s\in\lbrack
0,t\vee\bar{t}]}|\gamma_{t}(s\wedge t)-\bar{\gamma}_{\bar{t}}(s\wedge\bar
{t})|,\\
d_{\infty}(\gamma_{t},\bar{\gamma}_{\bar{t}}):=\sup_{0\leq s\leq t\vee\bar{t}%
}|\gamma_{t}(s\wedge t)-\bar{\gamma}_{\bar{t}}(s\wedge\bar{t})|+|t-\bar{t}|.
\end{array}
\]
It is obvious that $\Lambda_{t}$ is a Banach space with respect to $\Vert
\cdot\Vert$ and $d_{\infty}$ is not a norm.

\begin{definition}
A function $u:\Lambda\mapsto\mathbb{R}$ is said to be $\Lambda$--continuous at
$\gamma_{t}\in\Lambda$, if for any $\varepsilon>0$ there exists $\delta>0$
such that for each $\bar{\gamma}_{\bar{t}}\in\Lambda$ with $d_{\infty}%
(\gamma_{t},\bar{\gamma}_{\bar{t}})<\delta$, we have $|u(\gamma_{t}%
)-u(\bar{\gamma}_{\bar{t}})|<\varepsilon$. $u$ is said to be $\Lambda
$--continuous if it is $\Lambda$--continuous at each $\gamma_{t}\in\Lambda$.
\end{definition}

\begin{definition}
Let $u:\Lambda\mapsto\mathbb{R}$ and $\gamma_{t}\in\Lambda$ be given. If there
exists $p\in\mathbb{R}^{d}$, such that
\[
u(\gamma_{t}^{x})=u(\gamma_{t})+\langle p,x\rangle+o(|x|)\ \text{as}%
\ x\rightarrow0,\ x\in\mathbb{R}^{d}.\ \
\]
Then we say that $u$ is (vertically) differentiable at $\gamma_{t}$ and denote
the gradient of $\tilde{D}_{x}u(\gamma_{t})=p$. $u$ is said to be vertically
differentiable in $\Lambda$ if $\tilde{D}_{x}u(\gamma_{t})$ exists for each
$\gamma_{t}\in\Lambda$. We can similarly define the Hessian $\tilde{D}%
_{xx}^{2}u(\gamma_{t})$. It is an $\mathbb{S}(d)$-valued function defined on
$\Lambda$, where $\mathbb{S}(d)$ is the space of all $d\times d$ symmetric matrices.
\end{definition}

For each $\gamma_{t}\in\Lambda$ we denote
\[
\gamma_{t,s}(r)=\gamma_{t}(r)\mathbf{1}_{[0,t)}(r)+\gamma_{t}(t)\mathbf{1}%
_{[t,s]}(r),\ \ r\in\lbrack0,s].
\]
It is clear that $\gamma_{t,s}\in\Lambda_{s}$.

\begin{definition}
For a given $\gamma_{t}\in\Lambda$ if we have
\[
u(\gamma_{t,s})=u(\gamma_{t})+a(s-t)+o(|s-t|)\ \text{as}\ s\rightarrow
t,\ s\geq t,\ \
\]
then we say that $u(\gamma_{t})$ is (horizontally) differentiable in $t$ at
$\gamma_{t}$ and denote $\tilde{D}_{t}u(\gamma_{t})=a$. $u$ is said to be
horizontally differentiable in $\Lambda$ if $\tilde{D}_{t}u(\gamma_{t})$
exists for each $\gamma_{t}\in\Lambda$.
\end{definition}

\begin{definition}
Define $\mathbb{C}^{j,k}(\Lambda)$ as the set of function $u:=(u(\gamma
_{t}))_{\gamma_{t}\in\Lambda}$ defined on $\Lambda$ which are $j$ times
horizontally and $k$ times vertically differentiable in $\Lambda$ such that
all these derivatives are $\Lambda$--continuous.
\end{definition}

The following It\^{o} formula was firstly obtained by Dupire \cite{Dupire.B}
and then generalized by Cont and Fourni\'{e}, \cite{Cont-1},\ \cite{Cont-2}
and \cite{Cont-3}.

\begin{theorem}
\label{w2 copy(1)}Let $(\Omega,\mathcal{F},(\mathcal{F}_{t})_{t\in\lbrack
0,T]},P)$ be a probability space, if $X$ is a continuous semi-martingale and
$u$ is in $\mathbb{C}^{1,2}(\Lambda)$, then for any $t\in\lbrack0,T)$,
\end{theorem}

\[%
\begin{split}
u(X_{t})-u(X_{0})  &  =\int_{0}^{t}\tilde{D}_{s}u(X_{s})\,ds+\int_{0}%
^{t}\tilde{D}_{x}u(X_{s})\,dX(s)\\
&  \text{ \ \ }+\frac{1}{2}\int_{0}^{t}\tilde{D}_{xx}^{2}u(X_{s})\,d\langle
X\rangle(s),\quad\quad\ P-a.s.
\end{split}
\]

\subsection{The Fr\'{e}chet Derivatives}

Let $C^{\ast}$ and $C^{\dagger}$ be the space of bounded linear functionals
$\Phi:C\rightarrow R$ and bounded blinear functionals $\tilde{\Phi}:C\times
C\rightarrow R$, of the space $C,$ respectively. They are equipped with the
operator norms which will be, respectively, denoted by $\Vert\cdot\Vert^{\ast
}$ and $\Vert\cdot\Vert^{\dagger}$.

Fix $t\in\lbrack0,T)$. Let $B_{t}=\{ \upsilon1_{\{t\}},\upsilon\in R^{n}\}$,
where $1_{\{t\}}:[0,T]\rightarrow R$ is defined by%

\[%
\begin{array}
[c]{c}%
1_{\{t\}}(s):=\left\{
\begin{array}
[c]{l}%
0,\text{ \ for }s\in\lbrack0,t),\\
1,\text{ \ for }s=t,\\
0,\text{ \ for }s\in(t,T].
\end{array}
\right.
\end{array}
\]

We define the direct sum%
\[%
\begin{array}
[c]{c}%
C\oplus B_{t}:=\{ \phi(\cdot)+\upsilon1_{\{t\}}\mid\phi(\cdot)\in
C,\upsilon\in R^{n}\}
\end{array}
\]

and equip it with the norm $\Vert\cdot\Vert$ defined by%

\[%
\begin{array}
[c]{c}%
\Vert\phi(\cdot)+\upsilon1_{\{t\}}\Vert=\sup_{s\in\lbrack0,T]}\mid\phi
(s)\mid+\mid\upsilon\mid,\text{ \ \ }\phi(\cdot)\in C,\upsilon\in R^{n}.
\end{array}
\]

For each $\gamma(\cdot)\in C$ we denote%

\[%
\begin{array}
[c]{c}%
\gamma_{t}(s)=\left\{
\begin{array}
[c]{l}%
\gamma(s),\text{ \ \ }s\leq t,\\
\gamma(t),\text{ \ \ }s>t.
\end{array}
\right.
\end{array}
\]

It is clear that $\gamma_{t}(\cdot)\in C$.

\begin{definition}
We call a\ functional $\Psi:[0,T]\times C\longmapsto R$ is\ non-anticipative,
if for any $t\in\lbrack0,T]$ and $x(\cdot),$ $y(\cdot)\in C$ satisfying the
condition%
\[%
\begin{array}
[c]{c}%
y(\tau)=x(\tau)\text{ \ for }\tau\in\lbrack0,t],
\end{array}
\]
there holds the equality
\[%
\begin{array}
[c]{c}%
\Psi(t,x(\cdot))=\Psi(t,y(\cdot)).
\end{array}
\]

\end{definition}

\begin{definition}
\label{defi-1}$\forall\gamma(\cdot)\in C,$ if we have
\[
\Psi(s,\gamma_{t}(\cdot))=\Psi(t,\gamma(\cdot))+a(s-t)+o(|s-t|)\ \text{as}%
\ s\rightarrow t,\ s\geq t,\ \
\]
then we say that $\Psi(t,\gamma(\cdot))$ is differentiable at $t$ and denote
$D_{t}\Psi(t,\gamma(\cdot))=a$.
\end{definition}

$\Psi$ is said to be differentiable in $[0,T)$ if $D_{t}\Psi(t,\gamma(\cdot))$
exists for each $(t,\gamma(\cdot))\in\lbrack0,T]\times C.$

\begin{definition}
For a given $\gamma(\cdot)\in C$ and a non-anticipative $\Psi$, if we have
\[
\Psi(t,\varphi(\cdot))=\Psi(t,\gamma(\cdot))+D_{x}\Psi(t,\gamma(\cdot
))((\varphi(\cdot)-\gamma(\cdot)))+o(\Vert(\varphi(\cdot)-\gamma
(\cdot))1_{[0,t]}\Vert),
\]
\ for each $\varphi(\cdot)\in C,$ then we say that $\Psi(t,\gamma(\cdot))$ is
Fr\'{e}chet differentiable at $\gamma(\cdot)$.
\end{definition}

$\Psi$ is said to be differentiable in $C$ if $D_{x}\Psi(t,\gamma(\cdot))$
exists for each $(t,\gamma(\cdot))\in\lbrack0,T)\times C.$

\begin{remark}
For a non-anticipative $\Psi$, if $\Psi(t,\gamma(\cdot))$ is Fr\'{e}chet
differentiable at $\gamma(\cdot),$ then it is obvious
\[
D_{x}\Psi(t,\gamma(\cdot))(\eta(\cdot))=D_{x}\Psi(t,\gamma(\cdot))(\eta
(\cdot)1_{[0,t]}),\text{ \ \ }\forall\eta(\cdot)\in C.
\]

\end{remark}

\begin{definition}
Define $C^{j,k}([0,T)\times C)$ as the set of non-anticipative\ functions
$\Psi$ defined on $[0,T]\times C$ which are $j$ times differentiable in time
and $k$ times Fr\'{e}chet differentiable in $C$ such that all these
derivatives are continuous.
\end{definition}

Using similar techniques as in {\normalsize Mohammed} \cite{Mohammed1} and
\cite{Mohammed}, we have the following lemma.

\begin{lemma}
\label{extension-1} Suppose a non-anticipative $\Phi:[0,T]\times C\rightarrow
R$ is second order continuous differentiable. Then $\forall\phi(\cdot)\in C$,
the Fr\'{e}chet derivatives $D_{x}\Phi(t,\phi(\cdot))$ and $D_{xx}^{2}%
\Phi(t,\phi(\cdot))$ have unique weakly continuous linear and bilinear
extensions%
\[
\overline{D_{x}\Phi(t,\phi(\cdot))}\in(C\oplus B_{t})^{\ast},\ \overline
{D_{xx}^{2}\Phi(t,\phi(\cdot))}\ \in(C\oplus B_{t})^{\dagger}.
\]

\end{lemma}

\begin{proof}
It is sufficient to consider the one-dimensional case, i.e., $n=1$.

For a fixed $t\in\lbrack0,T)$ and $\phi(\cdot)\in C$, we will show that there
is a unique weakly continuous extension $\overline{D_{x}\Phi(t,\phi(\cdot
))}\in(C\oplus B_{t})^{\ast}$ of the first Fr\'{e}chet derivatives $D_{x}%
\Phi(t,\phi(\cdot))$. In other words, if $\{ \xi^{k}\}$ is a bounded sequence
in $C$ such that $\xi^{k}(s)\rightarrow\xi(s)$ as $k\rightarrow\infty$ for all
$s\in\lbrack0,T]$ where $\xi\in C\oplus B_{t},$ then $D_{x}\Phi(t,\xi
^{k}(\cdot))\rightarrow\overline{D_{x}\Phi(t,\xi(\cdot))}$ as $k\rightarrow
\infty.$ Note that $\Phi$ is non-anticipative. Then for all $\eta\in C,$%

\[%
\begin{array}
[c]{rl}%
D_{x}\Phi(t,\phi(\cdot))(\eta(\cdot))= & D_{x}\Phi(t,\phi_{t}(\cdot
))(\eta(\cdot)1_{[0,t]}).
\end{array}
\]
\

By the Riesz representation theorem, there is a unique finite Borel measure
$\mu$ on $[0,T]$ such that
\begin{equation}%
\begin{array}
[c]{c}%
D_{x}\Phi(t,\phi(\cdot))(\eta(\cdot))=\int_{0}^{t}\eta(s)d\mu(s).
\end{array}
\label{Riese-1}%
\end{equation}
Define $\overline{D_{x}\Phi(t,\phi(\cdot))}\in(C\oplus B_{t})^{\ast}$ by
\[%
\begin{array}
[c]{c}%
\overline{D_{x}\Phi(t,\phi(\cdot)+\upsilon1_{\{t\}})}=D_{x}\Phi(t,\phi
(\cdot))+\upsilon\mu(t),\text{ \ \ }\eta\in C,\text{ }\upsilon\in R.
\end{array}
\]
We know that $\overline{D_{x}\Phi(t,\phi(\cdot))}$ is weakly continuous by
Lebesgue's dominated convergence theorem. The weak extension $\overline
{D_{x}\Phi(t,\phi(\cdot))}$ is unique because for any $\upsilon\in R,$ the
function $\upsilon1_{\{t\}}$ can be approximated weakly by a sequence of
continuous functions $\{\xi_{0}^{k}\},$ where%

\[
\xi_{0}^{k}(s):=\left\{
\begin{array}
[c]{c}%
(ks+1)\upsilon,-\frac{1}{k}+t\leq s\leq t\\
\text{ \ \ \ }0,\text{ \ \ \ }0\leq s<-\frac{1}{k}+t.
\end{array}
\right.
\]

Similarly, we can construct a unique weakly continuous bilinear extension
$\overline{D_{xx}^{2}\Phi(t,\phi(\cdot))}\in(C\oplus B_{t})^{\dagger}$ for any
continuous bilinear form $D_{xx}^{2}\Phi(t,\phi(\cdot)).$
\end{proof}

\section{The relation between Dupire derivatives and Fr\'{e}chet derivatives}

In order to establish the relation between Dupire derivatives and Fr\'{e}chet
derivatives, we need the following auxiliary stochastic functional
differential equation: for given $t\in\lbrack0,T)$ and $\gamma(\cdot
)\in\Lambda_{T}$,
\begin{align}
&  dX^{\gamma_{t}}(s)=b(s,X^{\gamma_{t}}(\cdot))ds+\sigma(s,X^{\gamma_{t}%
}(\cdot))dW(s),\text{ \ }s\in\lbrack t,T],\label{SDE_1}\\
&  X^{\gamma_{t}}(r)=\gamma_{t}(r),\quad\ \ r\in\lbrack0,t],\nonumber
\end{align}
where $\{W(s),s\in\lbrack0,T]\}$ is the $d$-dimensional standard Brownian
motion; the process $\{X^{\gamma_{t}}(s),0\leq s\leq T\}$ takes values in
$\mathbb{R}^{n}$; $b:[0,T]\times{C}\rightarrow\mathbb{R}^{n}$ and$\ \sigma
:[0,T]\times{C}\rightarrow\mathbb{R}^{n}\times\mathbb{R}^{d}$ are
non-anticipative functionals.

\begin{definition}
A process $\{X^{\gamma_{t}}(s),$ $s\in\lbrack t,T]\}$ is said to be a strong
solution of the equation (\ref{SDE_1}) on the interval $[t,T]$ and through the
initial datum $\gamma_{t}\in{\Lambda}$ if it satisfies the following conditions:
\end{definition}

\noindent(1) $X_{t}^{\gamma_{t}}=\gamma_{t}$;\newline\ \ (2) $X^{\gamma_{t}%
}(s)$ is $\mathcal{F}(s)$-measurable for each $s\in\lbrack t,T]$%
;\newline\ \ (3) The process $\{X^{\gamma_{t}}(s),s\in\lbrack t,T]\}$ is
continuous and it satisfies the following stochastic integral equation
$P-a.s.$
\[
X^{\gamma_{t}}(s)=\gamma_{t}(t)+\int_{t}^{s}b(r,X^{\gamma_{t}}(\cdot
))dr+\int_{t}^{s}\sigma(r,X^{\gamma_{t}}(\cdot))dW(r).
\]

We assume $b,\sigma$ satisfy the following Lipschitz and bounded conditions.

\begin{assumption}
\label{assu-1} $b(\cdot,x(\cdot)),\sigma(\cdot,x(\cdot))$ are progressively
measurable processes for each $x(\cdot)\in{C}$,\ and there exists a constant
$c>0$ such that
\[
\mid b(s,x^{1}(\cdot))-b(s,x^{2}(\cdot))\mid+\mid\sigma(s,x^{1}(\cdot
))-\sigma(s,x^{2}(\cdot))\mid\leq c\parallel x_{s}^{1}(\cdot)-x_{s}^{2}%
(\cdot)\parallel,
\]
$\forall(s,x^{1}(\cdot)),(s,x^{2}(\cdot))\in\lbrack0,T]\times{C}$.
\end{assumption}

\begin{assumption}
\label{assu-2}There exists a constant $K>0$ such that
\[
\mid b(s,\Phi(\cdot))\mid+\mid\sigma(s,\Phi(\cdot))\mid\leq K,\quad
\forall(s,\Phi(\cdot))\in\lbrack0,T]\times{C}.
\]

\end{assumption}

Then we have the following theorem (see \cite{Lipster}):

\begin{theorem}
Under assumptions (\ref{assu-1}) and (\ref{assu-2}), the equation
(\ref{SDE_1}) has a unique strong solution.
\end{theorem}

By similar analysis as in {\normalsize Mohammed} \cite{Mohammed1} and
\cite{Mohammed}, we have the following result.

\begin{theorem}
\label{Tto-1}Let Assumptions (\ref{assu-1}) and (\ref{assu-2}) hold true.
$X^{\gamma_{t}}(\cdot)$ is the solution of (\ref{SDE_1}). Suppose a
non-anticipative $\Phi$ belongs to $C^{1,2}([0,T)\times C)$. Then for given
$\gamma\in C,$%
\begin{equation}%
\begin{array}
[c]{rl}%
\lim_{\varepsilon\rightarrow0^{+}}\frac{E[\Phi(t+\varepsilon,X^{\gamma_{t}%
}(\cdot))]-\Phi(t,\gamma(\cdot))}{\varepsilon}= & D_{t}\Phi(t,\gamma
(\cdot))+\overline{D_{x}\Phi(t,\gamma(\cdot))}(b(t,\gamma(\cdot))1_{\{t\}})\\
& +\frac{1}{2}\sum\limits_{j=1}^{n}\overline{D_{xx}^{2}\Phi(t,\gamma(\cdot
))}(\sigma(t,\gamma(\cdot))e_{j}1_{\{t\}},\sigma(t,\gamma(\cdot))e_{j}%
1_{\{t\}}),
\end{array}
\label{Ito-1}%
\end{equation}

\end{theorem}

\begin{proof}
Step 1.

Fix $\gamma(\cdot)\in C.$ Since $\Phi\in C^{1,2}([0,T]\times C),$ by Taylor's
theorem, for $\varepsilon>0$,
\[%
\begin{array}
[c]{rl}%
\Phi(t+\varepsilon,X^{\gamma_{t}}(\cdot))-\Phi(t,\gamma(\cdot))= &
\Phi(t+\varepsilon,\gamma_{t}(\cdot))-\Phi(t,\gamma(\cdot))+\Phi
(t+\varepsilon,X^{\gamma_{t}}(\cdot))-\Phi(t+\varepsilon,\gamma_{t}(\cdot))\\
= & D_{t}\Phi(t,\gamma(\cdot))\cdot\varepsilon+D_{x}\Phi(t+\varepsilon
,\gamma_{t}(\cdot))((X^{\gamma_{t}}(\cdot)-\gamma_{t}(\cdot
))1_{[0,t+\varepsilon]})\\
& +R(\varepsilon)+o(\varepsilon),\;a.s.
\end{array}
\]
where%
\[%
\begin{array}
[c]{cl}%
R(\varepsilon):= & \int_{0}^{1}(1-u)D_{xx}^{2}\Phi(t+\varepsilon,\gamma
_{t}(\cdot)+u\cdot(X^{\gamma_{t}}(\cdot)-\gamma_{t}(\cdot)))\\
& \text{ \ }((X^{\gamma_{t}}(\cdot)-\gamma_{t}(\cdot))1_{[0,t+\varepsilon
]},(X^{\gamma_{t}}(\cdot)-\gamma_{t}(\cdot))1_{[0,t+\varepsilon]})du.
\end{array}
\]

Taking expectation and dividing by $\varepsilon$, we have%

\begin{equation}%
\begin{array}
[c]{cl}%
\frac{E[\Phi(t+\varepsilon,X^{\gamma_{t}}(\cdot))]-\Phi(t,\gamma(\cdot
))}{\varepsilon}= & D_{t}\Phi(t,\gamma(\cdot))+D_{x}\Phi(t+\varepsilon
,\gamma_{t}(\cdot))\cdot E[\frac{1}{\varepsilon}(X^{\gamma_{t}}(\cdot
)-\gamma_{t}(\cdot))1_{[0,t+\varepsilon]}]\\
& +\frac{1}{\varepsilon}ER(t)+o(1).
\end{array}
\label{Ito-2}%
\end{equation}
Note that%

\[%
\begin{array}
[c]{rl}%
\lim_{\varepsilon\rightarrow0^{+}}[E\{ \frac{1}{\varepsilon}(X^{\gamma_{t}%
}(\cdot)-\gamma_{t}(\cdot))1_{[0,t+\varepsilon]}\}](s)= & \left\{
\begin{array}
[c]{l}%
\lim_{\varepsilon\rightarrow0^{+}}\frac{1}{\varepsilon}\int_{t}^{t+\varepsilon
}E[b(u,X^{\gamma_{t}}(\cdot))]du,\text{ }s=t\\
0,\text{ \ \ }0\leq s<t
\end{array}
\right. \\
= & b(t,\gamma(\cdot))1_{\{t\}},\text{ \ \ \ \ \ }0\leq s\leq t.
\end{array}
\]
Since $b$ is bounded, $\Vert E\{ \frac{1}{\varepsilon}(X^{\gamma_{t}}%
(\cdot)-\gamma_{t}(\cdot))1_{[0,t+\varepsilon]}\} \Vert_{C}$ is bounded at $t$
and $\gamma_{t}(\cdot)\in C.$ Hence%

\[%
\begin{array}
[c]{c}%
\lim_{\varepsilon\rightarrow0^{+}}[E\{ \frac{1}{\varepsilon}(X^{\gamma_{t}%
}(\cdot)-\gamma_{t}(\cdot))1_{[0,t+\varepsilon]}\}]=b(t,\gamma(\cdot
))1_{\{t\}}.
\end{array}
\]
Therefore, by Lemma (\ref{extension-1}) and the continuity of $D_{x}\Phi$ at
$\gamma(\cdot),$ we obtain%

\[%
\begin{array}
[c]{rl}
& \lim_{\varepsilon\rightarrow0^{+}}D_{x}\Phi(t+\varepsilon,\gamma_{t}%
(\cdot))[E\{ \frac{1}{\varepsilon}(X^{\gamma_{t}}(\cdot)-\gamma_{t}%
(\cdot))1_{[0,t+\varepsilon]}\}]\\
= & \lim_{\varepsilon\rightarrow0^{+}}D_{x}\Phi(t,\gamma(\cdot))[E\{ \frac
{1}{\varepsilon}(X^{\gamma_{t}}(\cdot)-\gamma_{t}(\cdot))1_{[0,t+\varepsilon
]}\}]\\
= & \overline{D_{x}\Phi(t,\gamma(\cdot))}(b(t,\gamma(\cdot))1_{\{t\}}).
\end{array}
\]

\noindent Step 2.

Finally we calculus the limit of the third term in the right-hand side of
(\ref{Ito-2}) as $\varepsilon\rightarrow0^{+}$. By the martingale property of
the It\^{o} integral and the Lipschitz continuity of $D_{xx}^{2}\Phi$, we have
the following estimates:%
\[%
\begin{array}
[c]{ll}
& \mid\frac{1}{\varepsilon}ED_{x}^{2}\Phi(t+\varepsilon,\gamma_{t}%
(\cdot)+u\cdot(X^{\gamma_{t}}(\cdot)-\gamma_{t}(\cdot)))((X^{\gamma_{t}}%
(\cdot)-\gamma_{t}(\cdot))1_{[0,t+\varepsilon]},(X^{\gamma_{t}}(\cdot
)-\gamma_{t}(\cdot))1_{[0,t+\varepsilon]})\\
& -\frac{1}{\varepsilon}ED_{xx}^{2}\Phi(t,\gamma(\cdot))((X^{\gamma_{t}}%
(\cdot)-\gamma_{t}(\cdot))1_{[0,t+\varepsilon]},(X^{\gamma_{t}}(\cdot
)-\gamma_{t}(\cdot))1_{[0,t+\varepsilon]})\mid\\
\leq & (E\Vert D_{xx}^{2}\Phi(t+\varepsilon,\gamma_{t}(\cdot)+u\cdot
(X^{\gamma_{t}}(\cdot)-\gamma_{t}(\cdot)))-D_{xx}^{2}\Phi(t,\gamma
(\cdot))\Vert^{2})^{\frac{1}{2}}[\frac{1}{\varepsilon^{2}}E\Vert(X^{\gamma
_{t}}(\cdot)-\gamma_{t}(\cdot))1_{[0,t+\varepsilon]}\Vert^{4}]^{\frac{1}{2}}\\
\leq & K(\varepsilon^{2}+1)^{\frac{1}{2}}(E\Vert D_{xx}^{2}\Phi(t+\varepsilon
,\gamma_{t}(\cdot)+u\cdot(X^{\gamma_{t}}(\cdot)-\gamma_{t}(\cdot)))-D_{xx}%
^{2}\Phi(t,\gamma(\cdot))\Vert^{2})^{\frac{1}{2}},
\end{array}
\]
where $t\in R^{+}$, $\gamma(\cdot)\in C$ and $K$ is a positive constant
independent of $u$. The last line tends to $0,$ uniformly for $u\in
\lbrack0,1],$ as $\varepsilon\rightarrow0^{+}$. Because $\Phi\in
C^{1,2}([0,T)\times C)$ and is bounded on $C$, we have the following weak
limit:
\[%
\begin{array}
[c]{rl}%
\lim_{\varepsilon\rightarrow0^{+}}\frac{1}{\varepsilon}ER(\varepsilon)= &
\int_{0}^{1}(1-u)\lim_{\varepsilon\rightarrow0^{+}}\frac{1}{\varepsilon
}ED_{xx}^{2}\Phi(t,\gamma(\cdot))((X^{\gamma_{t}}(\cdot)-\gamma_{t}%
(\cdot))1_{[0,t+\varepsilon]},(X^{\gamma_{t}}(\cdot)-\gamma_{t}(\cdot
))1_{[0,t+\varepsilon]})\\
= & \frac{1}{2}\sum\limits_{j=1}^{n}\overline{D_{xx}^{2}\Phi(t,\gamma(\cdot
))}(\sigma(t,\gamma(\cdot))e_{j}1_{\{t\}},\sigma(t,\gamma(\cdot))e_{j}%
1_{\{t\}}).
\end{array}
\]

\end{proof}

Note that $b$ and $\sigma$\ are non-anticipative functionals. We can rewrite
$b(t,\gamma(\cdot))=\tilde{b}(\gamma_{t})$ and $\sigma(t,\gamma(\cdot
))=\tilde{\sigma}(\gamma_{t})$, $\forall$ $\gamma(\cdot)\in C$.

\begin{corollary}
\label{w2} Let Assumptions (\ref{assu-1}) and (\ref{assu-2}) hold true.
$X^{\gamma_{t}}(\cdot)$ is the solution of (\ref{SDE_1}). $\Phi$ in
$\mathbb{C}^{1,2}(\Lambda)$ is non-anticipative. Then for any $t\in
\lbrack0,T)$,
\end{corollary}

\begin{equation}%
\begin{array}
[c]{rl}%
\lim_{\varepsilon\rightarrow0^{+}}\frac{E[\Phi(X_{t+\varepsilon}^{\gamma_{t}%
})]-\Phi(\gamma_{t})}{\varepsilon}= & \tilde{D}_{t}\Phi(\gamma_{t}%
)+\langle\tilde{D}_{x}\Phi(\gamma_{t}),\tilde{b}(\gamma_{t})\rangle+\frac
{1}{2}\langle\tilde{D}_{xx}\Phi(\gamma_{t})\tilde{\sigma}(\gamma_{t}%
),\tilde{\sigma}(\gamma_{t})\rangle.
\end{array}
\label{Ito-3}%
\end{equation}

It is easy to prove this corollary by Theorem \ref{w2 copy(1)}.

Now we build the relation between Fr\'{e}chet derivatives and Duprie derivatives.

\begin{theorem}
\label{w3}Suppose (i) $\Phi\in\mathbb{C}^{1,2}(\Lambda)$. (ii) When the domain
of $\Phi$ is limited to $[0,T]\times C$, it is non-anticipative and belongs to
$C^{1,2}([0,T)\times C)$. Then, for any given $\gamma(\cdot)\in C,$ we have
the following equalities:%
\[%
\begin{array}
[c]{rl}%
\tilde{D}_{t}\tilde{\Phi}(\gamma_{t})= & D_{t}\tilde{\Phi}(\gamma_{t}),\\
\tilde{\mu}(t)= & \tilde{D}_{x}\tilde{\Phi}(\gamma_{t}),\\
\tilde{\lambda}(t)= & \tilde{D}_{xx}^{2}\tilde{\Phi}(\gamma_{t}),
\end{array}
\]
where $\tilde{\Phi}(\gamma_{t})=\Phi(t,\gamma(\cdot))$, $\tilde{\mu}$\ and
$\tilde{\lambda}$ are the corresponding Borel measures of $\overline
{D_{x}\tilde{\Phi}(\gamma_{t})}$ and $\overline{D_{xx}^{2}\tilde{\Phi}%
(\gamma_{t})}$.
\end{theorem}

\begin{proof}
For given $\gamma(\cdot)\in C$, we rewrite $\tilde{\Phi}(X_{s}^{\gamma_{t}}%
)$=$\Phi(s,X^{\gamma_{t}}(\cdot)),\tilde{b}(\gamma_{t})=b(t,\gamma(\cdot))$
and $\tilde{\sigma}(\gamma_{t})=\sigma(t,\gamma(\cdot))$. By Theorem
(\ref{Tto-1}), we have%
\begin{equation}%
\begin{array}
[c]{rl}%
\lim_{\varepsilon\rightarrow0^{+}}\frac{E[\Phi(t+\varepsilon,X^{\gamma_{t}%
}(\cdot))]-\Phi(t,\gamma(\cdot))}{\varepsilon}= & \lim_{\varepsilon
\rightarrow0^{+}}\frac{E[\tilde{\Phi}(X_{t+\varepsilon}^{\gamma_{t}}%
)]-\tilde{\Phi}(\gamma_{t})}{\varepsilon}\\
= & \Phi(t,\gamma(\cdot))+\overline{D_{x}\Phi(t,\gamma(\cdot))}(b(t,\gamma
(\cdot))1_{\{t\}})\\
& +\frac{1}{2}\sum\limits_{j=1}^{n}\overline{D_{xx}^{2}\Phi(t,\gamma(\cdot
))}(\sigma(t,\gamma(\cdot))e_{j}1_{\{t\}},\sigma(t,\gamma(\cdot))e_{j}%
1_{\{t\}})\\
= & D_{t}\tilde{\Phi}(\gamma_{t})+\overline{D_{x}\tilde{\Phi}(\gamma_{t}%
)}(\tilde{b}(\gamma_{t})1_{\{t\}})\\
& +\frac{1}{2}\sum\limits_{j=1}^{n}\overline{D_{xx}^{2}\tilde{\Phi}(\gamma
_{t})}(\tilde{\sigma}(\gamma_{t})e_{j}1_{\{t\}},\tilde{\sigma}(\gamma
_{t})e_{j}1_{\{t\}}).
\end{array}
\label{Ito-4}%
\end{equation}

Similar as the proof of lemma (\ref{extension-1}), we know there is a unique
finite Borel measure $\tilde{\mu}$ on $[0,T]$ such that
\begin{equation}%
\begin{array}
[c]{c}%
D_{x}\tilde{\Phi}(\gamma_{t})(\eta(s))=\int_{0}^{t}\eta(s)d\tilde{\mu}(s).
\end{array}
\label{Rise-2}%
\end{equation}
Then we have%

\[
\overline{D_{x}\tilde{\Phi}(\gamma_{t})}(\tilde{b}(\gamma_{t})1_{\{t\}}%
)=\langle\tilde{\mu}(t),\tilde{b}(\gamma_{t})\rangle,
\]

There is also a unique finite Borel measure $\tilde{\lambda}$ on $[0,T]$ such that%

\[%
\begin{array}
[c]{c}%
\frac{1}{2}\langle\tilde{\lambda}(t)\tilde{\sigma}(\gamma_{t}),\tilde{\sigma
}(\gamma_{t})\rangle=\frac{1}{2}\sum\limits_{j=1}^{n}\overline{D_{xx}%
^{2}\tilde{\Phi}(\gamma_{t})}(\tilde{\sigma}(\gamma_{t})e_{j}1_{\{t\}}%
,\tilde{\sigma}(\gamma_{t})e_{j}1_{\{t\}}).
\end{array}
\]
It yields that%

\begin{equation}%
\begin{array}
[c]{rl}%
\lim_{\varepsilon\rightarrow0^{+}}\frac{E[\Phi(t+\varepsilon,X^{\gamma_{t}%
}(\cdot))]-\Phi(t,\gamma(\cdot))}{\varepsilon}= & \lim_{\varepsilon
\rightarrow0^{+}}\frac{E[\tilde{\Phi}(X_{t+\varepsilon}^{\gamma_{t}}%
)]-\tilde{\Phi}(\gamma_{t})}{\varepsilon}\\
= & D_{t}\tilde{\Phi}(\gamma_{t})+\langle\tilde{\mu}(t),\tilde{b}(\gamma
_{t})\rangle+\frac{1}{2}\langle\tilde{\lambda}(t)\tilde{\sigma}(\gamma
_{t}),\tilde{\sigma}(\gamma_{t})\rangle.
\end{array}
\label{Ito-5}%
\end{equation}
By Corollary (\ref{w2}), we have%
\begin{equation}%
\begin{array}
[c]{rl}%
\lim_{\varepsilon\rightarrow0^{+}}\frac{E[\Phi(X_{t+\varepsilon}^{\gamma_{t}%
})]-\Phi(\gamma_{t})}{\varepsilon}= & \tilde{D}_{t}\Phi(\gamma_{t}%
)+\langle\tilde{D}_{x}\Phi(\gamma_{t}),\tilde{b}(\gamma_{t})\rangle+\frac
{1}{2}\langle\tilde{D}_{xx}\Phi(\gamma_{t})\tilde{\sigma}(\gamma_{t}%
),\tilde{\sigma}(\gamma_{t})\rangle.
\end{array}
\label{Ito-6}%
\end{equation}

Notice that $b$ and $\sigma$ can take any values which satisfy Assumptions
(\ref{assu-1}) and (\ref{assu-2}). Comparing (\ref{Ito-5}) and (\ref{Ito-6}),
we obtain the results.
\end{proof}

\textbf{Acknowledgement.} \textit{The authors would like to thank Shige Peng
for pointing out that based on our results, Dupire derivative is a concept
weak than Fr\'{e}chet derivative. }

\end{document}